\documentclass[12pt]{amsart}
\usepackage{latexsym}
\usepackage{amsmath, amscd}

{\theoremstyle{plain}
    \newtheorem{Thm}{\bf Theorem}[section]
    \newtheorem{Prop}[Thm]{\bf Proposition}

    \newtheorem{Cor}[Thm]{\bf Corollary}
    \newtheorem{Q}[Thm]{\bf Question}

}
{\theoremstyle{remark}
    \newtheorem{Rem}[Thm]{\bf Remark}
    
}
{\theoremstyle{definition}
    \newtheorem{Def}[Thm]{\bf Definition}
}
\numberwithin{equation}{section}

\newcommand{\mlabel}[1]%
  {\mbox{}\marginpar{\raggedleft\hspace{0pt}{\rm\ttfamily#1}}\label{#1}}

\newcommand{\into}{\operatorname{\hookrightarrow}}

\newcommand{\Fdepth}{\operatorname{F-depth}}

\newcommand{\Min}{\operatorname{Min}}

\newcommand{\fm}{{\mathfrak m}}

\newcommand{\fn}{{\mathfrak n}}

\newcommand{\fA}{{\mathcal A}}
\newcommand{\fB}{{\mathcal B}}
\newcommand{\fC}{{\mathcal C}}

\newcommand{\ringR}{\text{$(R,\fm,k)$ }}
\newcommand{\ringS}{\text{$(S,\fn,k)$}}

\newcommand{\Max}{{\rm Max}}

 \newcommand{\height}{{\rm ht}}
\newcommand{\coheight}{{\rm coht}}
\newcommand{\bx}{{\bf x}}
 \newcommand{\ann}{{\rm Ann}}
 \newcommand{\Fann}{\operatorname{F-ann}}
 
\newcommand{\Dim}{{\rm dim}}

\newcommand{\Soc}{{\rm Soc}}

\newcommand{\Spec}{{\rm Spec}}

\newcommand{\hr}{{H_{\fm}^{d}(R)}}

\newcommand{\brq}{^{[q]}}
\newcommand{\inc}{\subseteq}

\newcounter{hours}\newcounter{minutes}

\newcommand{\excise}[1]{}
\begin{document}

\title{\bf Local cohomology and F-stability}
\dedicatory{Dedicated to Paul C.~Roberts}

\author[F.~Enescu]{Florian Enescu}

\address{Department of Mathematics and Statistics, Georgia State University, Atlanta, GA 30303}
\email{fenescu@gsu.edu}
\thanks{The author was partially supported by the NSA Young Investigator Grant H98230-07-1-0034.}
\date{}

\maketitle

\begin{abstract}
We study the relationship between the Frobenius stability of an Artinian module over an F-injective ring
and its stable part.
\end{abstract}

\section{Introduction}

Let $\ringR$ be a local Noetherian ring of positive characteristic
$p$, where $p$ is prime. Let $F : R \to R$ be the Frobenius
homomorphism on $R$, that $F(r) =r^p$, for all $r \in R$. This
homomorphism induces a natural Frobenius action on the local
cohomology modules $H^i_{\fm}(R)$, $i =0, \ldots, d= \dim(R)$. This action
is an effective tool in the study of local cohomology modules
as it was shown over the years by many authors. Our paper deals with
the concept of Frobenius stability which has its roots in an
influential paper by Hartshorne and Speiser~\cite{Ha} via the stable
part of a module $M$ endowed with a Frobenius action. Since then,
papers by Lyubeznik~\cite{L, L2}, Fedder and Watanabe~\cite{FW},
Enescu~\cite{E}, Singh and Walther~\cite{SW} and Sharp~\cite{Sh}
have explored various properties of local cohomology where Frobenius
stability played a role in some fashion. The work of Hartshorne and
Speiser, Lyubeznik, Singh and Walther dealt with the concept of
Frobenius depth and applications to the Grothendieck vanishing
problem (see~\cite{L2}, page 1), while Fedder and Watanabe, Enescu
and Sharp have explored connections to tight closure theory from a
different perspective.

Our main goal is to present a coherent description of Frobenius
stability and establish a clear relationship between the F-stability
in the sense of Fedder-Watanabe and the stable part of the local
cohomology as in Hartshorne-Speiser and Lyubeznik. With this goal
in mind, we will carefully discuss various aspects of Frobenius
stability and their relevance to the aforementioned papers hoping to
make our exposition valuable to the reader interested
in a unitary presentation of these aspects.

Before stating our main contributions, we need to introduce a few
notations and related concepts which are used throughout our
paper.  Everywhere in this note $\ringR$ will denote a local Noetherian
ring of positive characteristic $p$, where $p$ is prime, and Krull
dimension $d$. A Frobenius action on an Artinian $R$-module $M$ is
an additive map $F_M: M \to M$ such that $F_M(rm) = r^p F_M(m)$. We
will often drop the subscript "$-_M$" from our notation when there
is no danger of confusion. The main example of Artinian $R$-modules
that we will consider is that of local cohomology modules of $R$
with support in the maximal ideal $\fm$. Let $\underline{x} = x_1,
\ldots, x_d$ be a system of parameters for $R$. The $i$th local
cohomology module of $R$ equal the $i$th cohomology module of the
C\v{e}ch complex

$$ 0 \to R \to \oplus_{i=1}^d R_{x_i} \to
 \cdots \to \oplus_{i=1}^d
R_{\bx _{\hat{i}}} \to R_{\bx} \to 0,$$

where $\bx _{\hat{i}} = x_1 \cdots x_{i-1} x_{i+1}\cdots x_d$, and
$\bx= x_1 \cdots x_d$.

The Frobenius $F$ acts on $R$ and its localizations and hence
induces an action on the cohomology modules of the C\v{e}ch complex. We
will denote this induced Frobenius action on $H^i_{\fm}(R)$ by $F$.

The particular case of $\hr$ is very important, as this is the only
nonzero local cohomology module of $R$ when $R$ is Cohen-Macaulay.
An element of $\eta \in \hr$ is denoted by $\eta = [ \frac{z}{\bx
^s}]$ and the Frobenius action $F$ sends $\eta$ to $F(\eta) = [
\frac{z^p}{\bx ^{ps}}]$.

The following alternate description of $\hr$ will be useful later in
the paper. Let $x_1, \ldots, x_d$ be a system of parameters in $R$.
The local cohomology module $\hr$ can be obtained as a direct limit
of $R/(x_1^t,\ldots, x_d^t)$ where the maps of the direct system are
given by

$$R/(x_1^t,\ldots, x_d^t) \buildrel\bx^{l-t}\over\to R/(x_1^l,\ldots, x_d^l),$$
where $l \geq t$ and $\bx = x_1 \cdots x_d$. With this in mind,
an element $\eta \in \hr$ will be described by $\eta = [ z + (x_1^t,
\ldots, x_d^t)]$, where $z \in R$. The Frobenius action $F$ sends
$\eta$ to $F(\eta) = [z^p +(x_1^{pt}, \ldots, x^{pt}_d)]$.

\begin{Def}
{\rm Let $\ringR$ be a local Noetherian ring of positive
characteristic $p$, p prime. Then $R$ is} F-injective {\rm if $F$
acts injectively on $H^i_{\fm}(R)$ for all $i$.}
\end{Def}

In 1989 Fedder and Watanabe defined the notion of F-stability
for local cohomology modules of $R$ with support in the maximal
ideal $\fm$ of $R$ and studied it in the case of F-injective rings, ~\cite{FW}.
This definition can be extended to an Artinian $R$-module $M$
endowed with a Frobenius action. For such modules, Hartshorne and
Speiser have defined in 1977~\cite{Ha}, in the case when $R$ contains a
coefficient field $k$, a natural $k$-vector space $M_s \subseteq M$
called the {\it stable part} of $M$. See Section 2 for precise
definitions.  Our main contribution in this paper is to establish a
clear relationship between the F-stability of an Artinian module $M$
and its stable part.

More precisely, we prove the following:

\begin{Thm}
\label{intromain}
 Let $\ringR$ be a Noetherian local ring containing a
coefficient field $k$. Let $M$ be an Artinian $R$-module which
admits an injective Frobenius action.

Then $M$ is $F$-stable if and only if $M_s \neq 0$.

\end{Thm}

This result allows us to establish a connection between the set of
prime ideals $P$ in $R$ for which $R_P$ is F-stable and a set of
primes discovered by Lyubeznik in his work on $F$-modules (see Definition 2.1 for the definition
of F-stability for a local ring). These
primes are naturally related to the notion of F-depth and we
explore this relationship. This is done in Section 3. Section 4
presents a counterexample to a natural question on the behavior of
F-injectivity under flat local maps with regular fibers and shows
that complete F-injective $1$-dimensional domains with algebraically closed
residue field are regular.

We would like to review some of the basic definitions and facts from tight
closure theory that will be needed in our paper.

We use $q$ to denote a power of $p$, so $q = p^e$ for $e \ge 0$.
For $I \inc R$ set $I\brq = (i^q: i\in I)$.  Let $R^{\circ}$ be the complement in $R$
of the minimal primes of $R$.  We say that x belongs to the {\it tight closure} of
$I$ and write $x \in I^*$ if there exists $c \in R^{\circ}$ such that for all $q \gg 0$, $c x^q \in I\brq$.
We say that $x$ is in the {\rm Frobenius closure} of $I$, $I^F$, if there
exists a $q$ such that $x^q \in I\brq$, and say that $I$ is {\it Frobenius
closed} if $I = I^F$. When $R$ is reduced then $R^{1/q}$ denotes the ring of $q$th
roots of elements of $R$.  When $R^{1/q}$ is module-finite over $R$, the ring $R$ is called
{\it F-finite}.  We call $R$ {\rm weakly F-regular} if every ideal of $R$
is tightly closed.  A weakly F-regular ring is always normal, and under mild
hypotheses is Cohen-Macaulay.  A ring $R$ is {\it F-regular} if every localization
of $R$ is weakly F-regular.

We call an ideal $I = (x_1, \ldots x_n)$ a {\it parameter ideal} if $\height (I) \ge n$. The ring
$R$ is {\it F-rational} if every parameter ideal is tightly closed.
We note that F-rational Gorenstein rings are F-regular.

A ring $R$ for which $F: R \to R$ is a pure homomorphism is called {\it F-pure}. An F-pure ring is
F-injective and moreover an excellent and reduced ring $R$ is F-pure if and only if $I^F =I$ for all ideals
$I$ in $R$.

When $R$ is local Cohen-Macaulay, then $R$ is F-injective
if and only if some (equivalently, every) ideal generated by a
system of parameters is Frobenius closed. Moreover, if $R$ is Cohen-Macaulay F-injective, then
$R_P$ is Cohen-Macaulay F-injective for any prime ideal $P$ in $R$. This fact is well known to the experts. A proof of it follows from the more general fact that the Frobenius closure 
of ideals generated by regular sequences commutes with localization. This result has a proof identical to 
that of Theorem 4.5 in~\cite{HH94} (take $c=1$) where it is shown that the tight closure of ideals generated by 
regular sequences commutes with localization. In fact, a more general theorem states that the F-injectivity property localizes for 
an F-finite local ring that admits a dualizing complex, see~\cite{Sc}, Proposition 4.3.

\section{Frobenius stability of Artinian modules}

Let $\ringR$ be a Noetherian local ring of characteristic $p$, where $p$ is a prime number, and dimension $d$. Let $M$
be an Artinian $R$-module.

Assume that $M$ admits a Frobenius action $F=F_M : M \to M$, i.e. an
additive map with the property that $F(rm) = r^pF(m)$ for all $m \in
M$ and $ r \in R$.

Let $\Soc_R(M)= \{ x \in M: \fm \cdot x =0 \}$. This is a
$R$-submodule of $M$ which is naturally a vector space over $R/\fm
=k$. In fact, $\Soc_R(M)$ is finite dimensional over $k$.

Our main examples of Artinian $R$-modules that admit a Frobenius action are  the local cohomology modules of $R$, $H^i_{\fm}(R)$ with
$0 \leq i \leq d$. However we find it helpful to present the notions related to Frobenius stability in the more general context of Artinian modules and then to apply them to local cohomology.

The following definition is inspired by Fedder and Watanabe who considered it only in the case of local cohomology modules.

\begin{Def}[Fedder-Watanabe]
Let $M$ be an Artinian $R$-module that admits a Frobenius action Let
$S=\Soc_R(M)$ be the socle of $M$. Denote $F^e(S) = \{F^e (m) : m
\in S \}.$

We say that $M$ is {\it $F$-unstable} if there exists $N >0$ such
that $S \cap F^e(S) = 0$ for all $e \geq N$. Note that the zero
module is F-unstable. If $M$ is not F-unstable, then it will be
called {\it $F$-stable.}

In general, we say that $R$ is {\it $F$-unstable} if $H^i_{\fm}(R)$
is {\it $F$-unstable} for every $i$.
\end{Def}

The reader should be aware that a submodule $N$ of $M$ is sometimes
called F-stable if $F(N) \subseteq N$. We will call such submodules
$N \subseteq M$ {\it F-invariant} to avoid any possible confusion.

The following reformulation can be established in the case of an
injective Frobenius action on $M$. We decided to include its proof
for the convenience of the reader.

\begin{Prop}(Fedder-Watanabe)
\label{fw}
Let $\ringR$ be a local ring. Let $M$ be an Artinian $R$-module that admits an injective Frobenius action $F: M \to M$.
Let $S=\Soc(M)$ be the socle of $M$.

If $S \cap F^e(S) \neq 0$
holds for infinitely many $e >0$ (i.e. $M$ is F-stable), then there exists $0 \neq \eta \in S$ such that $F^e( \eta ) \in S$ for every $e \geq 0$.
\end{Prop}

\begin{proof}

Assume that $S \cap F^e(S) =0$ for infinitely many $e$. Denote $M_e = S \cap F^e(S)$.

Claim: $M_{e+1} \subset F(M_e)$.

Indeed, take $ m \in M_{e+1}$, that is $m \in S$ and $m = F^{e+1} (a)$, for some $a \in S$.
So, $0 = \fm ^{[p^j]} m = F^j( \fm F^{e+1-j}(a))$ and therefore $F^{e+1-j}(a) \in S$, for all $1 \leq j \leq e+1$.

So $ F^e(a) \in S$ and obviously $F^e(a) \in F^e(S)$, since $a \in S$, proving the claim.

One can consider $M_e$ as a vector space over $R/\fm = k$ with the multiplication $l * m = l^{[p^e]} m$, for $l \in k$ and
$m \in M_e$.

Moreover, $F : M_e \to F(M_e)$ is in fact an injective $k$-linear map, and hence bijective. So
$\dim_k (M_e) = \dim_k (F(M_e)) $ and hence $\dim _k (M_{e+1}) \leq \dim_k (M_e)$ with equality if and only if
$F^e(M_e) = M_{e+1}$.

But from the finite dimensionality of $S$, we get that $\dim (M_e)$
is constant for $e$ large enough. Since $M_e \neq 0$ for infinitely
many $e$, we know that this constant must be non-zero. So, there
exist $N >0$ and nonzero $\gamma \in M_N$, where $\dim M_e = \dim M_N$
for all $e \geq N$ and $F^{l}(M_N) = M_{N+l}$ for all $l \geq 0$. We
have that, for all $e \geq N$,  $F^{e-N}$ is an isomorphism from
$M_N$ to $M_e$. Moreover, there exists $\eta \in S$ such that
$F^N(\eta) = \gamma$.

Let us check that $F^e(\eta) \in S$ for all $e$: if $e \leq N$, then let us first notice that $\fm \cdot \gamma = 0$.
so, $\fm \cdot F^N(\gamma) =0$. This implies that $F^{N-e}(\fm \cdot F^e (\gamma))=0$ and using the $F$-injective of $R$ we get our statement.
For $e \geq N$, $F^e(\eta) = F^{e-N}(\gamma) \in M_e \subset S$.  So, $F^e(\eta) \in S$, for all $e$.
\end{proof}

Let $m \in M$ and set $C_m = \langle F^e(m) : e \geq 0 \rangle _R$ the
$R$-submodule of $M$ generated by all $F^e(m)$ with $m \in M$.
Denote $\Fann(m) : = {\rm Ann}_R(C_m) = \{ r \in R : rF^e(m) =0, \
{\rm for \ all } \ e \geq 0 \}$.

A number of papers have considered the annihilators of F-invariant submodules from various points of view,~\cite{E, EH, Sh}. An important result, independently obtained by
Sharp on one hand and Enescu and Hochster on the other, explains the main properties of these ideals in the case of an injective
Frobenius action on $M$. We will quote this result in Theorem~\ref{seh} below.

A family $\Gamma$ of radical ideals is {\it closed under primary
decomposition} if for any ideal $I \in \Gamma$ and any irredundant
intersection $ I = P_1 \cap \cdots \cap P_n$, where $P_1,...,P_n$
are prime ideals, it follows that $P_i \in \Gamma$, for all
$i=1,...,n$.

\begin{Thm}[Sharp; Enescu-Hochster]
\label{seh} Let $\ringR$ be a Noetherian local ring of
characteristic $p$ and let $M$ be an Artinian $R$-module that admits
an injective Frobenius action denoted $F$. Denote $\Gamma = \{
\ann_R(N) : N \subseteq M \  \ such \ that \ F(N) \subseteq N \}$.

Then $\Gamma$ is a finite set of ideals, consists of radical ideals and is closed under primary decomposition.
\end{Thm}

\begin{Rem}
\label{charac}
{\rm  Under the conditions of Theorem~\ref{seh} (or~\ref{fw}), $M$
is F-stable if and only if $\fm \in \Gamma$ if and only if there
exists $0\neq m \in M$ such that $\fm = \Fann(m)$. This follows
immediately from Proposition~\ref{fw}.}
\end{Rem}

Now it is time to explain the relationship between
the F-stability of $M$ and a stability concept introduced by
Hartshorne and Speiser and refined by Lyubeznik.

\begin{Def}[Hartshorne-Speiser]
Assume that $R$ contains a coefficient field $k$. Let $M$ be an
Artinian $R$-module such that the Frobenius $F$ acts on $M$. For all
$j$, let $F^j(M) =\{ F^j(m) : m \in M \}$ and $\langle F^j(M) \rangle _k$ be the
$k$-vector space spanned by $F^j(M)$ in $M$. The {\it stable} part
of $M$ is $M_s: = \cap _{j \geq 1} \langle F^j(M) \rangle_k$.
\end{Def}

While this definition depends on the choice of the coefficient field $k$, Lyubeznik has shown that the dimension of $M_s$ as a $k$-vector
space is independent of $k$ (see~\cite{L}, Corollary 4.11). Moreover, Hartshorne and Speiser proved that, in the case that $k$ is perfect,
$M_s$ is finite dimensional over $k$ and the naturally induced Frobenius action is bijective on $M_s$.

Te be more precise, let $k \subset K$ and let $R^K = K \hat{\otimes}_k R$, the complete tensor product. The ring $R^K$ is complete and local,
with maximal ideal equal to $K \hat{\otimes}_k \fm$ and residue field $K$.

For an $R$-module $M$, we let $M^K = R^K \otimes_R M$. Since $M$ is Artinian over $R$, we have that $M^k= K \otimes_k M $.
If $M$ admits a Frobenius action $F_M : M \to M$, then
this induces a Frobenius action on $M^K$ by $F_{M^K} ( l \otimes m ) = l^p \otimes F_M(m),$ for $l \in K, m\in M$. Moreover,
$\Soc_{R^K}(M^K ) = K \otimes_k \Soc_R(M)$.

Let us denote $M_{nil} = \{ m \in M : {\rm there \ exists} \ e {\rm \
such \ that } \ F^e(m)=0\}$ and $M_{red}= M/M_{nil}$. Obviously, if $F$ acts
injectively on $M$ then $M_{nil}=0$. In general, $M_{nil}$ is invariant
under $F$ and $F$ acts injectively on $M_{red}$.

\begin{Thm}[Hartshorne-Speiser; Lyubeznik]
\label{hsl}
Assume that $R$ contains a coefficient field $k$ and let
$M$ be an Artinian $R$-module with a Frobenius action $F$ on it.
Then $M_s$ is finite dimensional over $k$, $F: M_s \to M_s$ is
injective and the $k$-vector subspace of $M$ spanned by $F(M_s)$
coincides with $M_s$. Moreover, if $K$ is the perfect closure of
$k$, $M^K$ is Artinian over $R^K$, Frobenius acts on it, $(M^K)_s =
K \otimes_k M_s$ and $\dim _k M_s= \dim _K (M^K)_s$.

\end{Thm}

\begin{Prop}[Hartshorne-Speiser]
\label{hs} Let $\ringR$ be a Noetherian local ring containing a
perfect coefficient field $k$.
\begin{enumerate}
\item
Let $M$ be an Artinian $R$-module and let $F=F_M$ be a Frobenius
action on $M$. If $F$ acts injectively on $M$ then $M_s \subseteq
\Soc_R(M)$.

\item
\noindent
Let $$0 \to N\to M \to L \to 0,$$ be a short exact
sequence of $R$-Artinian modules that admit compatible Frobenius
actions, i.e. the following diagram is commutative:

\[
\begin{CD}
0 @>>> N @>>> M @>>> L @>>> 0 \\
@.  @VVF_NV  @VVF_MV      @VVF_LV @. \\
0 @>>> N @>>> M @>>> L @>>> 0
\end{CD}
\]

\medskip

Moreover, assume that $F_L$ acts injectively on $L$.

Then $$ 0 \to N_s \to M_s \to L_s$$ is a short exact sequence of $k$-vector spaces.

\end{enumerate}

\end{Prop}

We are now in position to state the main result of this section.

\begin{Thm}
\label{main}
 Let $\ringR$ be a Noetherian local ring containing a
coefficient field $k$. Let $M$ be an Artinian $R$-module which
admits an injective Frobenius action.

Then $M$ is $F$-stable if and only if $M_s \neq 0$. Also, $M_s
\subseteq \Soc_R(M)$.

\end{Thm}

\begin{proof}
Assume that $M$ is F-stable. Then there exists $m \in M$ such that $\fm = \Fann(m)$, by Remark~\ref{charac}. Therefore, $$\fm \cdot F^e(m) =0,$$
for every $e \geq 0$.

Let $S=\Soc(M)$ and denote $M_e = S \ \cap \langle F^e(M) \rangle_k$.

Note that $M_e$ is a $k$-vector subspace of $S$. It can be easily checked that $M_{e+1} \subseteq M_e$ since $F^{e+1}(M) \subseteq F^e(M)$.

But $F^e(m) \in \langle F^e(M) \rangle_k$ and since $\fm \cdot F^e(m)=0$ we
conclude that $F^e(m) \in M_e$. Moreover, $F$ acts injectively on
$M$ hence $F^e(m) \neq 0$ for all $e$. So, $M_e \neq 0$.

Since $S$ is a finite dimensional $k$-vector space and $\{ M_e \}_e$ forms a descending chain of $k$-vector subspaces, there
exists $e_0$ such that $M_e =M_{e_0}$ for all $ e\geq e_0$. But $M_e \neq 0$ for all $e$. Hence

 $$ 0 \neq M_{e_0} = \cap_{e} M_e \subseteq \cap_e \langle F^e(M) \rangle_k=M_s,$$
proving the first part of the theorem.

Let $K$ be the perfect closure of $k$ and $M^K=  R^K \otimes _R M = K \otimes_k M$ where $R^K = R \hat{\otimes}_k K$.
We have that $M_s \subset (M^K)_s$.

Consider the following short exact sequence:

$$0 \to (M^K)_{nil} \to M^K \to M^K/(M^K)_{nil} \to 0,$$
where the Frobenius actions on each module are naturally compatible.
Note that the Frobenius action on $M^K/(M^K)_{nil}$ is injective.
Applying Proposition~\ref{hs} we get the following exact sequence

$$ 0 \to (M^K)_s \to (M^K/(M^K)_{nil})_s.$$

Since $F$ acts injectively on $M^K/(M^K)_{nil}$ then Propostion~\ref{hs}
(1) applies and we get $$(M^K/(M^K)_{nil})_s \subseteq
\Soc_{R^K}(M^K/(M^K)_{nil}).$$

Now assume $M_s \neq 0$ and let $\eta \in M_s$ arbitrarly chosen.
Then $1 \otimes \eta \in (M^K)_s$.

Using the inclusions

$$(M^K)_s \subseteq (M^K/(M^K)_{nil})_s \subseteq
\Soc_{R^K}(M^K/(M^K)_{nil}),$$ we obtain that $\fm ^K \cdot (1 \otimes
\eta) \in (M^K)_{nil}$.

However, $(1 \otimes \fm) \cdot (1 \otimes \eta) \subset \fm ^K
\cdot (1 \otimes \eta) \in (M^K)_{nil}.$ Hence there exists $e$ such
that $F^e ( 1 \otimes \fm \eta) = 0 \in M^K.$ But $F^e ( 1 \otimes
\fm \eta) = 1 \otimes F^e(\fm \eta)$.

Moreover, $M \subset M^K$ and hence $F^e( \fm \eta) = 0$ in $M$. But
$F$ acts injectively on $M$. Therefore, $\fm \eta =0$ or, equivalently, $\eta \in
\Soc(M)$. This shows that $M_s \subseteq \Soc(M)$.

To finish the proof, let $ 0 \neq \eta \in M_s$. Then since $F$ acts
injectively on $M$ we have that $0\neq F^e(\eta)$ and $F^e(\eta) \in
M_s$. So, $F^e(\eta) \in \Soc_R(M)$, or $\fm \cdot F^e(\eta) =0$ for
all $e$. Therefore $M$ is F-stable, since $\fm = \Fann(\eta)$.
\end{proof}

\begin{Cor}
Let $\ringR$ be a local Noetherian F-injective ring containing a coefficient field.

Then the following assertions are equivalent:

\begin{enumerate}
\item
$R$ is F-stable;
\item
there exists $i$ and $0 \neq \eta_i \in H^i_{\fm}(R)$ such that $\fm
= \Fann(\eta_i)$;
\item
there exists $i$ such $H^i_{\fm}(R)_s \neq 0$.
\end{enumerate}

\end{Cor}

\section{F-stable primes}

In this section we will apply the results of the previous section to
the top local cohomology module of a Cohen-Macaulay local ring.

Let $\ringR$ be a local Noetherian ring of positive characteristic
$p$ and dimension $d$. Then $R$ admits a natural Frobenius action $F
:R \to R$, defined by $F(r) = r^p$ for all $r \in R$. This action
induces a Frobenius action on the local cohomology modules
$H^i_{\fm}(R)$, $i = 0, \ldots, d.$

If $R$ is Cohen-Macaulay then there is only one nonzero local
cohomology module, namely $\hr$. As defined in Section 2, in this
case $R$ is F-stable if and only if $\hr$ is F-stable.

We will focus our attention now on two interesting sets of prime ideals in $R$ that one can
define in relation to the action of Frobenius on $\hr$.

\begin{Def}
\label{sets}
Let $R$ be a Cohen-Macaulay ring of positive characteristic $p$.

Let $$\fA=\fA(R) = \{ P \in \Spec(R) : R_P {\rm \ is } \ {\rm
F-stable} \}.$$

If $\ringR$ is local then let

$$\fB=\fB(R) = \{ P \in \Spec(R) :
{\rm there \ exists } \ \eta \in \hr {\rm \ such \ that} \ P =
\Fann(\eta) \}.$$

\end{Def}

We remark that Theorem~\ref{seh} implies that when $R$ is
F-injective and Cohen-Macaulay then $\fB(R)$ is finite. Also, when $R$ is F-injective and Cohen-Macaulay, $R$ is
F-stable if and only if $\fm \in \fB(R)$ (due to
Remark~\ref{charac}).

Let $x_1, \ldots, x_d$ be a system of parameters in $R$ and denote $I = (x_1, \ldots, x_d)$. The local
cohomology module $\hr$ can be obtained as a direct limit of
$R/(x_1^t,\ldots, x_d^t)$ where the maps of the direct system are
given by

$$R/(x_1^t,\ldots, x_d^t) \buildrel{\bx ^{l-t}}\over\to R/(x_1^l,\ldots, x_d^l),$$
where $l \geq t$ and $\bx = x_1\cdots x_d$.

Let $x \in R$ and define $\widetilde{I(x)} : =\{ c \in R : cx^q \in I\brq, {\rm
for \ all \ } q \geq 0\}$. Also, let $I(x) : =\{ c \in R : cx^q \in I\brq, {\rm
for \ all \ } q \gg 0\}$ as in Definition 1.1 in~\cite{E}. Clearly, $\widetilde{I(x)} \subseteq I(x)$.

Let $\overline{x}$ be the class of $x$ in $R/(x_1,\ldots, x_d)$ and
let $\eta \in \hr$ be the image of $\overline{x}$ via

$$R/(x_1,\ldots, x_d) \to \hr.$$

Then $\Fann(\eta) = \widetilde{I(x)}$ whenever $R$ is Cohen-Macaulay.

Conversely, consider $0 \neq \eta \in \hr$ and let $x$ and $t$ such
that $\eta$ is the image of $\overline{x} \in R/(x_1^t,...,x_d^t)$
under the natural inclusion
$$R/(x_1^t,\ldots, x_d^t) \to \hr.$$

Then let $J=I_t=(x_1^t, \ldots, x_d^t)$, and note that, when $R$ is Cohen-Macaulay, $J(x)=\widetilde{I_t(x)}=
\Fann (\eta)$.

\begin{Rem}
\label{alt}

(i) Using the notation just introduced we have
$$\fB (R) = \{ P \in \Spec(R) : I(x) = P  {\rm \ for \ some \ }
I=(x_1, \ldots, x_d) {\rm \ generated \ by \ parameters \ and  \ } x
\in R
 \}.$$

(ii) An important observation is that $I(x) = \widetilde{I(x)}$, whenever $R$ is F-injective and Cohen-Macaulay. 

For the proof of this claim, first note that $R$ is F-injective and Cohen-Macaulay then $J =J^F$, for any ideal $J$ generated by a system of 
parameters. Then $cx^{p^e} \in I^{[p^e]}$ implies that $c^px^{p^e} \in I^{[p^e]}$ 
or, equivalently, $(cx^{p^{e-1}})^p \in (I^{[p^{e-1}]})^{[p]}.$ But $I^{[p^{e-1}]}$ is an ideal generated
by a system of parameters, so it is Frobenius closed. Therefore, we deduce that $c x^{[p^{e-1}]} \in I^{[p^{e-1}]}$. This shows
that $cx^{p^e} \in I^{[p^e]}$ implies that $c x^{[p^{e-1}]} \in I^{[p^{e-1}]}$, which gives our claim.

Therefore, whenever $R$ is F-injective and Cohen-Macaulay, $I$ is an ideal generated by a system of parameters and $x \in R$, 
we can and will use the notation $I(x)$ versus $\widetilde{I(x)}$ without further explanation.

\end{Rem}

 These comments allow one to remark to state Propositions 2.6 and 2.7 proven
 in~\cite{E} in the following concise form. Please note that whenever $X$ is a partially order
set, the notation $\Max(X)$ refers to the maximal elements of $X$ under the order relation.

\begin{Thm}
\label{complete}
Let $\ringR$ be a local Cohen-Macaulay F-injective complete ring of
dimension $d$.

Then $\fB\subseteq \fA$ and $\Max(\fA) = \Max(\fB)$.
\end{Thm}

It is useful to know how F-stability behaves under a flat local ring
extension that has nice fibers.

\begin{Thm}
\label{flat} Let $\ringR \to \ringS$ be a flat local ring
homomorphism of Cohen-Macaulay F-injective rings.

Assume that the closed fiber $S/\fm S$ is regular and $R$ is
F-stable. Then $\fm S \in \Max(\fB(S)).$ \

In particular if $\ringR$ is local Cohen-Macaulay F-injective and
F-stable then $\widehat{R}$ is local Cohen-Macaulay F-injective and
F-stable.
\end{Thm}

\begin{proof}
Since $R$ is F-stable, then $\fm = \Fann (\eta) = I(u)$, for some
$I= (x_1, \ldots, x_d)$ and $u \in R$. Since $I(u) \subseteq I(ru)$,
for any $r \in R$, we can arrange that the image of $u$ in $R/I$
belongs to $\Soc(R/I)$.

Let $z_1,...,z_n$ be a regular system of parameters for $S/\fm S$.
Note that for any basis of $\Soc(R/I)$ say
$\overline{u_1},...,\overline{u_l}$, where $u_1, \ldots, u_l$ belong
to $R$, their images $\overline{u_1},...,\overline{u_l}$ in $S/(I,
z_1,\ldots, z_n)S$ form a basis for its socle.

Let $J = (x_1, \ldots, x_d, z_1, \ldots, z_n)S$. Note that $0 \neq
\overline{u} \in \Soc(S/J)$. It is clear that $\fm \subseteq J(u)$
and so $\fm S \subseteq J(u)$.

Let us consider $c \in S$ such that $c u^q \in J \brq S$ for all
$q$.

The induced ring homomorphism $R \to S/(z_1^q, \ldots, z_n^q)S$ is
still flat. Let $S_q =S/(z_1^q, \ldots, z_n^q)S.$

Let $c u^q \in J \brq S$ and map this further to $S_q$. Keeping the
same notations for convenience we see that $c \in (I\brq : _{S_q}
u^q)$. Since $R \to S_q$ is flat we get that $c \in (I\brq :_R
u^q)S_q$ for all $q$.

Now, $\fm = I(u) = \cap _q (I\brq : u^q) \subseteq I\brq :
u^q\subseteq \fm.$ So, $\fm = I\brq:u^q$ for all $q$.

In conclusion $c \in \fm S_q$, for all $q$. Pulling back to $S$ we
get that $c \in \fm S + (z_1^q, \ldots, z_n^q)S$ for all $q$. So, $
c\in \fm S$ by taking intersection over all $q$.

We have shown that $\fm S = J(u)$. This automatically proves the Theorem whenever $\dim (S /\fm S) = 0$ since
under our hypotheses this shows $\fn =\fm S = J(u)$. This gives that $S$ is F-stable or, in other words, 
$\fn \in \Max (\fB(S))$ (we use here that $S$ is F-injective and Cohen-Macaulay). In particular, the claim of the Theorem
regarding the completion of $R$ is also proved.

But now let $P = \fm S$ and consider $R \to S_P$. The argument presented above for the case of
zero-dimensional closed fiber shows that  $S_P$ is F-stable. As mentioned at the end of Section $1$, the F-injectivity property 
localizes, so $S_P$ is F-injective.

We have shown that $P= \fm S \in \fB (S)$. We can find a maximal 
element of $\fB (S)$ the form $J(w)$, with $w$ in $S$, containing $P= \fm S$. This follows from Theorem 2.1 in~\cite{E} (which rests on Lemma 1.4 
and Proposition 1.5 of the same paper).
We can assume that $w$ maps to an element in the socle of $S/J$. But this socle is generated
by the socle of $R/I$ via the natural map $R/I \to S/J$. So, we can in fact assume that $w$ belongs to $R$ such that its image in $R/I$ belongs to the socle of $R/I$.

Now, note that since $\fm \in J(w)$ we have $\fm w^{[q]} \in I\brq S_q$ for all $q$, using the same notations introduced 
above. Contracting back 
to $R$ and using that $R \to S_q$ is faithfully flat we get that $\fm w \brq \in I \brq$, for all $q$, which, in turn, 
shows that $\fm = I(w)$. By repeating
the argument displayed earlier we get $P = \fm S = J(w)$, which proves that $P= \fm S$ is a maximal element of $\fB(S)$.

\end{proof}

\begin{Rem}
In our previous Proposition it is necessary to assume that $S$ is
F-injective. This condition cannot be deduced form the other
hypotheses as our Proposition~\ref{example} shows.
\end{Rem}

In what follows we will present a connection between the prime ideals
discussed in Definition~\ref{sets} and a set of prime ideals
discovered by Lyubeznik in his work on Frobenius depth. So we will move our attention to the concept of Frobenius
depth considered by Hartshorne-Speiser, see page 60
in~\cite{Ha} and Lyubeznik, see Definition 4.12 in ~\cite{L} and
Definition 4.1 in~\cite{L2}.

Assume that $R$ is Noetherian and let $P \in \Spec(R)$. We set
$\coheight(P) = \Dim (R/P)$. Let $k=k(P)$ be the residue field of
$R_P$. Consider a copy of $k$ as a coefficient field of
$\widehat{R_P}$ and let $K=K(P)$ be its perfect closure. We let
$R(P)$ be $\widehat{R_P} \otimes_k K$ which is a local ring with a
perfect residue field. Note that $R(P)^K$ is the completion of $R(P)$ at the
maximal ideal $P \widehat{R_P} \otimes K$ (we now refer the reader to the complete tensor product
defined in Section 2). The ring $R(P)^K$ is Noetherian.

Clearly $H^i_{PR_P}(R_P)= H^i_{P \widehat{R_P}}(\widehat{R_P})$ and
$H^i_{PR(P)}(R(P)) =H^i_{P\widehat{R_P}}(\widehat{R_P}) \otimes_k
K$.

Also since $R(P)^K$ is the completion of $R(P) \otimes_k K$ we also
have that 

$$H^i_{PR(P)^K}(R(P)^K) = H^i_{PR(P)}(R(P)).$$

Since $H^i_P(\widehat{R_P})$ is Artinian we in fact have that
$$H^i_{P\widehat{R_P}}(\widehat{R_P})^{K}
=H^i_{P\widehat{R_P}}(\widehat{R_P}) \otimes_k K ,$$
which gives

$$H^i_{P\widehat{R_P}}(\widehat{R_P})^{K} =H^i_{PR(P)}(R(P)) = H^i_{PR(P)^K}(R(P)^K),$$
and hence are in position to apply results of Theorem~\ref{hsl}. In
conclusion, $H^i_{PR(P)}(R(P))_s \neq 0$ if and only if
$H^i_{P\widehat{R_P}}(\widehat{R_P})_s \neq 0$.

\begin{Def}[Hartshorne-Speiser]
Let $R$ be a Noetherian ring with $\dim (R) < \infty$. Using the
notation introduced above, we say that the Frobenius depth of $R$,
denoted $\Fdepth_{HS} (R)$, is

$ \Fdepth_{HS}(R) = \max \{ r : H^i_{PR(P)} (R(P))_s = 0 \ {\rm for \ all
\ } \ i < r -\coheight(P) , \ {\rm for \ all} \ P \in \Spec(R) \}$.

\medskip
 \noindent The considerations above allows us to simplify this to

$ \Fdepth_{HS}(R) = \max \{ r : H^i_{P\widehat{R_P}} (\widehat{R_P})_s = 0 \ {\rm for \ all \ } \ i < r -\coheight(P) , \ {\rm for \ all} \ P \in \Spec(R) \}$.

\end{Def}

To be able to parse through this concept more easily it is helpful
to first introduce a local concept of Frobenius depth.

\begin{Def}
Let $(R, \fm)$ be a local ring. Then the $\Fdepth$ of  $R$ at $\fm$, denoted  $ \Fdepth(\fm, R)$, is

$\Fdepth (\fm , R) = \max \{ r : H^i_{\fm } (\hat{R})_s = 0 \ {\rm
for \ all} \ i < r \} = \min \{ r : H^r_{\fm } (\hat{R})_s \neq 0
\}.$
\end{Def}

It is easy to conclude that

\begin{Rem}
$\Fdepth_{HS}(R) = \Min \{ \Fdepth (PR_P, R_P) + \coheight (P): P \in
\Spec(R) \}$.
\end{Rem}
Note that for the maximal ideal $\fm$ of $R$ we have $\coheight(\fm) =0$, so the inequality 
above implies that $\Fdepth (\fm, R) \leq \Fdepth_{HS}(R) $.

Whenever $R$ is zero dimensional, we can see that $H^0_{\fm } (R) = R$ and $R_s =k \neq 0$. So, in this case 
$\Fdepth (\fm, R) =0$.

Assume that $P$ is a minimal prime ideal of $R$, $R_P$ is a local ring of dimension zero and we 
obtain $0 \leq \Fdepth_{HS} (R) \leq \dim (R/P)$.

If $R$ is Cohen-Macaulay, then its localizations at prime ideals as well as their completions are
Cohen-Macaulay, so $\Fdepth (P, R_P)$ either equals $\dim (R_P)$ or $\infty$ for any prime ideal $P$ in $R$. This implies that 
$\Fdepth_{HS} (R) = \dim (R)$ if $R$ is Cohen-Macaulay.

Using his theory of $F$-modules, Lyubeznik has proven the following
theorem.
\begin{Thm}(Lyubeznik)
Let $R$ be a homomorphic image of a finite type algebra over a
regular local ring.

Then there exist only a finite number of prime ideals $P$ in $R$
such that $H_{P\widehat{(R_P)}}^i (\widehat{R_P}) _s \neq 0$.

\end{Thm}

\begin{proof}

Proposition 4.14 in~\cite{L} states that there exist only finitely
many prime ideals $P$ such that  $H^i_{P\widehat{R_P}}
(\widehat{R_P}^K)_s \neq 0$

But as noticed $H^i_{P\widehat{R_P}}
(\widehat{R_P}^K)_s \neq 0$ if and only if $H_{P\widehat{R_P}}^i (\widehat{R_P}) _s \neq 0$.

\end{proof}

This leads us to consider the following set.

\begin{Def}
Let $\ringR$ be a Noetherian ring of positive characteristic $p$.

Let $\fC(R) =\{ P \in \Spec(R) : \ {\rm there \ exists \ } h \ {\rm
such \ that \ } H_{P\widehat{(R_P)}}^h (\widehat{R_P}) _s \neq 0 \}$.

\end{Def}

The investigations of Section 1 allow us to state the following
result.

\begin{Prop}
Let $R$ be a Cohen-Macaulay F-injective ring.

Then $$\fA(R) = \fC(R).$$

\end{Prop}

\begin{proof}
Since the F-injectivity property localizes, we conclude that $R_P$
is F-injective for any prime ideal $P$ of $R$. Also, $R_P$ is
Cohen-Macaulay as well.

Let $P \in \fA(R)$, that is, $R_P$ is F-stable. Then $\widehat{R_P}$ is
F-stable as well. The ring $\widehat{R_P}$ is complete and we can apply
Theorem~\ref{main} to it. Assume that $\height (P) =h$. Then
$H^h_{P}(\widehat{ R_P})_s \neq 0$. Therefore, $P \in \fC(R)$.

Conversely, if for a prime ideal $P$ we have that $H^h_{P}(\widehat{
R_P})_s \neq 0$ then $h =\height (P)$ and $\widehat{R_P}$ is F-stable
as a consequence of Theorem~\ref{main}. This implies that $R_P$ is
F-stable and hence $P \in \fA(R)$.

\end{proof}

\begin{Cor}
Let $R$ be a Cohen-Macaulay F-injective ring. Assume that $R$ is a
homomorphic image of an algebra of finite type over a regular local
ring.

Then $\fA(R)$ is finite.

\end{Cor}

The concept of Frobenius depth can appear a little technical at a
first glance. Hartshorne and Speiser used it to give answers to an
important problem stated by Grothendieck: Let $A$ be a commutative
ring and let $I \subset A$ be an ideal. If $n$ is an integer, find
conditions under which $H^i_i(M) =0$ for all $i >n$ and all
$A$-modules $M$. For their answer this problem we refer the reader to~\cite{Ha}.

More recently Lyubeznik introduced the following variant of
Frobenius depth, Definition 4.1 in~\cite{L2} and proved that it
coincides with the earlier introduced concept of Hartshorne-Speiser
under mild conditions. 

\begin{Def}[Lyubeznik]
Let $\ringR$ be a local Noetherian ring of positive characteristic.
The $\Fdepth$ of $R$ is the smallest $i$ such that $F^s$ does not
send $H^i_{\fm}(R)$ to zero for any $s$. We will denote this number
by $\Fdepth_L(R)$.

\end{Def}

\begin{Thm}[Lyubeznik] Let $R$ be a local ring which is a homomorphic
image of a regular local ring. Then

$$\Fdepth_{HS} (R) = \Fdepth_L (R).$$

\end{Thm}

Lyubeznik has also given interesting characterizations for the cases
$\Fdepth_L(R) \leq 1$. Singh and Walther have added to these results 
by proving the following interesting theorem. We
will show later that one cannot replace the hypothesis $k$
algebraically closed and hope to obtain the same result.

\begin{Thm}[Singh-Walther]
\label{sw}
Let $\ringR$ be a complete local ring of positive
characteristic. Assume that the residue field $k$ is algebraically
closed. Then the number of connected components of the  punctered
spectrum $\Spec(R) \setminus \{\fm\}$ is
$$1 + \dim_k (H^1_{\fm}(R))_s.$$

\end{Thm}

\section{Examples}

The reader can note that the F-injectivity assumption is crucial in
our treatment of F-stability in Section~2. At times one needs to
enlarge the residue field of a ring to its perfect closure and this
brings up a natural question which is of interest in own right: is
F-injectivity preserved under those circumstances?

To be more precise, let us formulate the following question.

\begin{Q}
\label{behavior} {\rm Let $\ringR$ be a local Noetherian F-injective
ring of positive characteristic $p$. Assume that $k$ is a
coefficient field of $R$, that is the composition of the maps  $k
\into R \to R/\fm=k$ is the identity map. Let $k \to k'$ be a purely
inseparable extension. Is the ring $k' \otimes_k R$ F-injective?
Is $k' \ \widehat{\otimes}_k \ R$ F-injective?}
\end{Q}

Let $K$ be the perfect closure of $k$, that is $K = k^{1/p^\infty}$. The ring $S= K \otimes_k R$ is local 
with maximal ideal $\fm \otimes_k K$. By Dumitrescu (Theorem 4.8~\cite{D}), the ring 
$S$ is Noetherian (one can take the finite purely inseparable extension refered to in Theorem 4.8~\cite{D} to equal the 
trivial extension $k \subseteq k$, because the residue field of $R$ is $k$). Since $K \ \widehat{\otimes}_k \ R$ is the
completion of $S$ at the maximal ideal it is Noetherian as well. Also by flat base change we see that
$k' \otimes_k R \subset K \otimes_k R$ is a flat local extension and hence the ring $k' \otimes_k R$ is Noetherian as well.

Since F-injectivity commutes with completion (see for example Lemma 2.7 in~\cite{EH}) we will examine
whether $S$ is F-injective when $R$ is.

We will present an example that provides a negative answer to the
question. The question is however still open if we further assume
that $R$ is normal.

Let $k \subset L$ be an algebraic field extension and $x$ an
indeterminate. Consider $R =k+xL[[x]]$ which is a local Noetherian
with maximal ideal $\fm=xL[[x]]$ and residue field equal to $k$.
This ring is one dimensional complete domain. This example was considered in
~\cite{EH}, Example 2.16.

In~\cite{EH} it was proven that $R$ is F-injective if and only if
$$L^p \cap k = k^p.$$

Our claim is that for a suitable extension $k \subset L$, the ring $k^{1/p} \otimes_k R$ is not
F-injective. In fact, it is not even reduced.

Now note that $k^{1/p} \subset K$ and hence $k^{1/p} \otimes_k R \subset K \otimes_k R$ and therefore $K \otimes_k R$
cannot be reduced when $k^{1/p} \otimes_k R$ is not reduced.
\begin{Prop}
\label{example}
Let $k \subset L$ be a finite algebraic extension
such that $L^p \cap k =k^p$ and $k \subset L $
is not separable.

Let $x$ be an indeterminate and consider $R = k+xL[[x]]$.

\begin{enumerate}
\item
R is F-injective local complete F-stable 1-dimensional ring.
\item
$ R \otimes_k k^{1/p}$ is not reduced, hence not
F-injective.
\end{enumerate}

\end{Prop}

\begin{proof}
(1) It is immediate that $R$ is local, one-dimensional and complete.
For a proof of the F-injectivity our $R$ we refer to Example 2.16
in~\cite{EH}.

Note that $x$ is a parameter for $R$. Let $I=xR$. We will show that
there exist $u \in R$ such that $\fm = I(u) = \{ c \in R : cu^q \in
I\brq \ {\rm for \ all \ } q \}$. If this holds then, according to
Remark~\ref{alt}, $\fm \in \fB(R)$. But this is equivalent to the
F-stability of $R$, since $R$ is Cohen-Macaulay, by Remark~\ref{charac}
or Theorem~\ref{complete}.

Let $a \in L \setminus k$. Then $ax \in \fm= xL[[x]]$, but $ax$ is not
in $I$.

Let $c \in \fm$ arbitrary. Then $c (ax)^q $ is a formal power series
of order at least $q+1$ which means that it belongs to $I\brq = x^q
R$. So if we let $u =ax$ then  $\fm \subseteq I(u)$. But $I(u) = R$
if and only if $ u \in I^F =I$ which is not the case, as $a \in L
\setminus k$.

(2) Since $k \subseteq L$ is not separable, then $k^{1/p} \otimes_k L$ is not reduced.

Let $u = \sum_{i=1}^h a_i^{1/p} \otimes_k b_i$, where $a_i \in k, b_i \in L$, $ i=1,..., h$ such that
$u^p =0$ but $u \neq 0$. Note  that $u^p = \sum_{i=1}^h a_i \otimes_k b_i^p = 1 \otimes \sum_{i=1}^h a_ib_i^p$, which is
equivalent to $\sum_{i=1}^h a_ib_i^p =0$ in $L$.

Consider $v = \sum_{i=1}^h a_i^{1/p} \otimes_k (b_i \cdot x)$ as an element of $k^{1/p} \otimes_k R$. Now, $v^p =
1 \otimes [(\sum_{i=1}^h a_ib_i^p)x^p] = 0.$ Let us argue that $v \neq 0$. Regard $L\cdot x$ as a $k$-vector space 
and note that $k^{1/p} \otimes_k L \simeq k^{1/p} \otimes_k L\cdot x$ as $k$-vector spaces. Moreover,
$L \cdot x \subset R$ as $k$-vector spaces and hence  $k^{1/p} \otimes_k L \cdot x \subset k^{1/p} \otimes_k R$. Under the isomorphism
the nonzero element $u \in k^{1/p}\otimes_k L$ corresponds to $v \in k^{1/p} \otimes_k L \cdot x \subset k^{1/p} \otimes_k R$.  In conclusion,
$v$ is nonzero in $k^{1/p} \otimes_k R$ as well.

\end{proof}

We recall the Example 2.14 in~\cite{EH} which exhibits a finite extension $k \subset L$ such that $L^p \cap k =k^p$ 
and also that $k \subset L $ is not separable. Let $E$ be an infinite perfect field of characteristic $p$ and set $k=E(u,v)$. 
Let $L=k[y]/(y^2p+y^p u -v)$. In~\cite{EH} it was proven that $L^p \cap k =k^p$  and $[L: k(L^p)] \geq p >1.$
It is known that a finite extension $k \subset L$ is separable if and only if $k(L^p) =L$.

Using the notations in the preceding Proposition, it is interesting
to note that $R \to k^{1/p} \otimes_k R$ is a flat local map with closed fiber
equal to $k^{1/p}$, a field. In this example the F-injectivity
of $R$ does not pass to $k^{1/p} \otimes_k R$. Much of this has to do with the
fact that $k^{1/p}$ is not separable over $k$, where $k$ is
the residue field of $R$. Indeed, whenever the residue field extension is separable  F-injectivity is preserved under flat local maps.
The following result is essentially contained in~\cite{AE}, Theorem 4.2, although not stated for F-injective rings. 
We present the argument here in our context.

Let $R$ be a $k$-algebra where $k$ is a field. We say that $R$ is {\it geometrically F-injective} over $k$, if, for every finite field extension
$k \subseteq k'$ the ring $k' \otimes_k R$ is F-injective. Note that a field extension $L$ of $k$ is geometrically F-injective if and 
only if $k \subset L$ is separable.

\begin{Thm}[Aberbach-Enescu]
Let $(R, \fm, k) \to (S, \fn, L)$ be a flat local map such that $S/\fm S$ is Cohen-Macaulay and geometrically F-injective over $k$. If $R$ is 
F-injective and Cohen-Macaulay, then $S$ is F-injective and Cohen-Macaulay.
\end{Thm}

\begin{proof}
The proof follows as in Theorem 4.2 (1) in~\cite{AE} with minor changes: replace tight closure by Frobenius closure, use $c=1$ and let $I$
equal an ideal generated by a system of parameters. 
\end{proof}

We would like to end by showing that Theorem~\ref{sw} cannot be
extended to rings with non-algebraically closed residue field. Also a result on one-dimensional
F-injective rings is provided as well.

\begin{Rem}
{\rm Let $k \subseteq L$ be a field extension that satisfies the
conditions of Prop~\ref{example}. Then $R=k + xL[[x]]$ is a domain
and hence the punctured spectrum is connected, while $\dim_k
H^1_{\fm}(R)_s \neq 0$. The latter claim follows since $R$ is
F-injective and F-stable hence by Theorem~\ref{main} $H^1_{\fm}(R)_s
\neq 0$.}

\end{Rem}

\begin{Prop}
Let $\ringR$ be a local complete F-injective ring of dimension $1$.
Assume that $k$ is algebraically closed and $R$ is domain. Then $R$
is regular.
\end{Prop}

\begin{proof}

The punctured spectrum $\Spec(R) \setminus \{ \fm \}$ is connected
since $R$ is domain. According to Theorem~\ref{sw}, we have that
$H^1_{\fm}(R)_s = 0$. But Theorem~\ref{main} implies that $R$ is not
F-stable. This implies that $R$ is F-rational, hence regular as
$\dim(R) =1$.

\end{proof}

{\bf Acknowledgment:} The author thanks Mel Hochster for discussions on local cohomology and suggesting considering
Example 2.16 in relation to the question in Section 4. The author also thanks the anonymous referee for his/her helpful
comments.

\end{document}